\documentclass[12pt,letterpaper]{amsart}
\usepackage{euler, epic,eepic,latexsym, amssymb, amscd, amsfonts, xypic, url, color, epsfig}
\usepackage{epigraph}
 \newlength{\baseunit}               % the basic unit length
 \newcount{\numlines}                % depth of picture (in number of lines)
\newtheorem{tm}{Theorem}[section]
\newtheorem{pr}[tm]{Proposition}
\newtheorem{lm}[tm]{Lemma}
\newtheorem{co}[tm]{Corollary}
\newtheorem{df}[tm]{Definition}
\newtheorem{rmk}[tm]{Remark}

\newcommand{\cat}[1]{{\mathbf{#1}}}

\newcommand{\GW}{\mathrm{GW}}

\newcommand{\rH}{\operatorname{H}}

\newcommand{\Z}{\mathbb{Z}}

\newcommand{\Q}{\mathbb{Q}}

\newcommand{\kbar}{\overline{k}}
\newcommand{\R}{\mathbb{R}}

\newcommand{\C}{\mathbb{C}}

\newcommand{\End}{\operatorname{End}} % Ring of endomorphisms
\newcommand{\Gal}{\operatorname{Gal}}

\newcommand{\sSet}{\cat{sSet}}

\newcommand{\Sheaves}{\cat{Sh}}
\newcommand{\Sm}{\cat{Sm}} 
 
\newcommand{\spaces}[1]{\cat{sPre}(\Sm_{#1})} 
\newcommand{\Spt}{\cat{Spt}} % added TBJW 5 March, for spectra

\newcommand\bbb[1]{\ensuremath{{\mathbf{#1}}}}

\newcommand{\assoc}{\theta}
\newcommand{\A}{\operatorname{A}}
\newcommand{\Tr}{\operatorname{Tr}}

\newcommand{\Et}{\operatorname{Et}}

\newcommand{\Spec}{\operatorname{Spec}}
\newcommand{\Rep}{\operatorname{Rep}}
\newcommand{\Burn}{\operatorname{Burn}}
\newcommand{\GL}{\operatorname{GL}}

\newcommand{\Ker}{\operatorname{Ker}}
\newcommand{\Image}{\operatorname{Image}}
\newcommand{\Sphere}{\bbb{S}}

\newcommand{\et}{{\text{\'et}}}

\newcommand{\calH}{{ \mathcal H}}

\newcommand{\hidden}[1]{\footnote{Hidden:  #1}}
\renewcommand{\hidden}[1]{}

\begin{document}
\pagestyle{plain}
\title{An \'Etale Realization which does NOT exist}

\author{Jesse Leo Kass}

\address{Current: J.~L.~Kass, Dept.~of Mathematics, University of South Carolina, 1523 Greene Street, Columbia, SC 29208, United States of America}
\email{kassj@math.sc.edu}
\urladdr{http://people.math.sc.edu/kassj/}

\author{Kirsten Wickelgren}

\address{Current: K.~Wickelgren, School of Mathematics, Georgia Institute of Technology, 686 Cherry Street, Atlanta, GA 30332-0160}
\email{kwickelgren3@math.gatech.edu}
\urladdr{http://people.math.gatech.edu/~kwickelgren3/}


\date{\today}

\begin{abstract}
For a global field, local field, or finite field $k$ with infinite Galois group, we show that there can not exist a functor from the Morel--Voevodsky $\bbb{A}^1$-homotopy category of schemes over $k$ to a genuine Galois equivariant homotopy category satisfying a list of hypotheses one might expect from a genuine equivariant category and an \'etale realization functor. For example, these hypotheses are satisfied by genuine $\Z/2$-spaces and the $\bbb{R}$-realization functor constructed by Morel--Voevodsky. This result does not contradict the existence of \'etale realization functors to (pro-)spaces, (pro-)spectra or complexes of modules with actions of the absolute Galois group when the endomorphisms of the unit is not enriched in a certain sense. It does restrict enrichments to representation rings of Galois groups.
\end{abstract}
\maketitle
%\tableofcontents

{\parskip=12pt % closing bracket is just before the bibliography 
%\refstepcounter{section}

\section{Introduction}
%\epigraph{le motif associ\'e \`a une vari\'et\'e alg\'ebrique constituerait l'invariant cohomologique ``ultime", ``par excellence", dont tous les autres (associ\'es aux diff\'erentes th\'eories cohomologiques possibles) se d\'eduiraient, comme autant d~``incarnations" musicales, ou de ``r\'ealisations" diff\'erentes.}{A. Grothendieck}

Grothendieck envisioned that cohomological invariants of algebraic varieties are controlled by certain motives, with transcendental invariants, such as the Galois action on \'etale cohomology, being recovered by realization functors. Work of Morel--Voevodsky on $\bbb{A}^1$-homotopy theory provides candidate categories of motives, and realization functors have been constructed from the Morel--Voevodsky $\bbb{A}^1$-homotopy category and its stabilization \cite{morelvoevodsky1998} \cite{Isaksen-etale_realization} \cite{Quick-stable_realization} \cite{Ayoub-realization_etale}, \cite[Section~7.2]{CD-etale_motives}. 

For example, Morel--Voevodsky construct an $\bbb{R}$-realization functor \cite[3.3]{morelvoevodsky1998}. On the level of schemes, this functor takes an $\bbb{R}$-scheme $X$ to the topological space (or simplicial set) of its complex points $X(\bbb{C})$ together with the action of the Galois group $\Gal(\bbb{C}/\bbb{R})$ on $X(\bbb{C})$. The association $X \mapsto X(\bbb{C})$ gives rise to functors from the $\bbb{A}^1$-homotopy categories of spaces and of $\bbb{P}^1$-spectra. The target category of this functor may be taken to be the genuine $\Gal(\bbb{C}/\bbb{R})$-homotopy category in the former case and the genuine $\Gal(\bbb{C}/\bbb{R})$-equivariant stable homotopy category in the latter. The adjective genuine (or {\em fine} in the terminology of \cite[3.3]{morelvoevodsky1998}) refers to the homotopy theory in which a weak equivalence of $\Gal(\bbb{C}/\bbb{R})$-spaces $X \to Y$ is not only an equivariant map which is a weak equivalence in the non-equivariant sense, but also satisfies the requirement that the map on fixed points $X^H \to Y^H$ is an equivalence for all subgroups of the group acting on the spaces, in this case $\Gal(\bbb{C}/\bbb{R})$ \cite{LewisMaySteinberger}.  In the stable case,  more notation would be required to define genuine $\Gal(\bbb{C}/\bbb{R})$-equivariant spectra, but one consequence of the more stringent notion of weak equivalence is that the group of endomorphisms of the sphere spectrum encodes interesting information about $\Gal(\bbb{C}/\bbb{R})$-sets, in contrast to $\bbb{Z}$, which is isomorphic to the endomorphisms of the sphere in ordinary stable homotopy. It follows that the Euler characteristic in genuine equivariant homotopy theory can be enriched from an element of $\bbb{Z}$ to an element of the Burnside ring $\A(\Gal(\bbb{C}/\bbb{R}))$ of formal differences of finite $\Gal(\bbb{C}/\bbb{R})$-sets. 

One might hope to generalize this construction to fields $k$ with infinite Galois groups. Namely, let $G=\Gal_k=\Gal(\kbar/k)$ denote the absolute Galois group of $k$. One could hope to construct a category of genuine $G$-spectra or $G$-pro-spectra receiving an \'etale realization functor from $\bbb{P}^1$-spectra over $k$. Indeed, for profinite groups, homotopy theories of genuine $G$-(pro)-spectra are constructed in \cite{Fausk} and \cite{BarwickI}, and \cite{Quick-Profinite_G-spectra} contains a construction of a homotopy theory of profinite $G$-equivariant spectra. \'Etale realization functors have also been constructed over $k$ \cite{Isaksen-etale_realization} \cite{Quick-stable_realization} \cite{Ayoub-realization_etale}, but not so that the target is a genuine $G$-equivariant homotopy theory. 

We show here that there is a reason for this lacuna. Namely, for $k$ a global field, local field, or finite field with infinite Galois group, there does not exist a symmetric monoidal \'etale realization functor from $\bbb{P}^1$-spectra over $k$ to a genuine $G$-equivariant homotopy theory under a set of hypothesis on what the terms {\em \'etale realization} and {\em genuine} are expected to imply.  The hypotheses are best explained by the proof, and the proof is based on a result of Hoyois \cite[Theorem 1.5]{Hoyois_lef}. Namely, Hoyois shows that the $\bbb{A}^1$-Euler characteristic of a smooth proper variety with trivial cotangent sheaf is $0$, which implies, in particular, that the $\bbb{A}^1$-Euler characteristic of an elliptic curve $E$ is $0$. The outline of the proof is as follows. Suppose there is a symmetric monoidal \'etale realization functor $L \Et$. Then the Euler characteristic of $L \Et (E)$ is well-defined and equal to $0$. On the other hand, in genuine $G$-spectra, one could expect that the Euler characteristic is related to the alternating sum of the cohomology groups (with coefficients in some ring or field) considered with their $G$-actions. For example, this is the case for $G$ a finite group (Proposition \ref{assocchiX=repchi}). The theory of weights for Galois representations shows that the alternating sum of the \'etale cohomology groups for $E$ is non-zero in a representation ring. Since one could expect the cohomology groups of the \'etale realization to be the \'etale cohomology groups, at least for finite coefficients (as in \cite[Proposition 5.9]{Friedlander}), this is a contradiction.

A precise formulation is as follows. First, we introduce a definition. Let $k$ be a field and let $R$ be a ring. Let $\operatorname{Funct}(\Gal_k, R)$ denote the ring of functions $\Gal_k \to R$ with point-wise addition and multiplication.

\begin{df}\label{df:representation_ring}
By a {\em representation ring} of $\Gal_k$ with coefficients in $R$, we mean a ring $\Rep( \Gal_k, R)$ such that: 
\begin{enumerate}
\item The isomorphism class of a finitely generated $R$-module $A$ with a continuous $\Gal_k$-action given by $A = \rH^i_{\et}(X_{\kbar}, R)$, where $X$ is a smooth, proper variety over $k$, determines an element $[A]$ of $\Rep( \Gal_k, R)$. 
\item The trace map which takes $A$ to the function $\Gal_k \to R$ defined by taking $g$ in $G$ to the trace of $g$ acting on $A$, $$g \mapsto \Tr g \vert A$$ extends to a homomorphism $\Rep(\Gal_k, R) \to \operatorname{Funct}(\Gal_k, R)$.
\end{enumerate}
\end{df}

\begin{tm}\label{into:mainthm}

Let $k$ be a global field, local field with infinite Galois group, or finite field. Let $p$ be a prime different from the characteristic of $k$. It is impossible to simultaneously construct all of the following \begin{enumerate}
\item \label{mainthm_HgalK} A symmetric monoidal category $\calH_{\Gal_k}(\Spt)$, enriched over abelian groups. Let $\A(\Gal_k)= \End_{\calH_{\Gal_k}(\Spt)}(\Sphere)$ denote the endomorphisms of the symmetric monoidal unit $\Sphere$. 

\item \label{mainthm_Rep} For all $n$, a representation ring $\Rep( \Gal_k, \Z/{p^n})$ with coefficients in $\Z/{p^n}$. 

\item \label{mainthm_realization} A symmetric monoidal additive functor $L \Et: \calH_{\bbb{A}^1}(\Spt^{\bbb{P}^1}(k)) \to \calH_{\Gal_k}(\Spt)$, and ring homomorphisms $\assoc_n: \A(G) \to \Rep(\Gal_k, \Z/p^n)$ such that $$\assoc_n \chi (L\Et \Sigma^{\infty}_{\bbb{P}^1}X_+) = \sum_{i=-\infty}^{\infty} (-1)^i [H^i_{\et} (X_{\kbar}, \Z/p^n)]$$ for smooth proper schemes $X$ over $k$.%, or even just for $X$ an elliptic curve over $k$.
\end{enumerate}
\end{tm}

We are thinking of $\calH_{\Gal_k}(\Spt)$ as a genuine $\Gal_k$-equivariant stable homotopy category, $\A(\Gal_k)$ as a Burnside ring of $\Gal_k$, $\Rep( \Gal_k, \Z/{p^n})$ as a representation ring of $\Z/(p^n)$-modules with $\Gal_k$-action, and $L \Et$ as an \'etale realization functor. Since there are interesting contexts in which one may require a continuous Galois representation to satisfy additional conditions (e.g., be de Rham, crystalline, semi-stable, potentially semi-stable, etc.), we do not require that $\Rep( \Gal_k, \Z/{p^n})$ be the Grothendieck ring of continuous Galois representations, although this is certainly a possibility for \eqref{mainthm_Rep}. Theorem \ref{into:mainthm} is proven as Theorem \ref{tm:nonexistence} below and there is another variant given in Theorem \ref{tm:nonexistencev2}, where the representation rings involved only have coefficients in fields.

In Section \ref{GR-realization}, we show that the constructions of Theorem \ref{into:mainthm} do exist when $k=\bbb{R}$, the realization functor \eqref{mainthm_realization} being that of Morel--Voevodsky \cite[3.3]{morelvoevodsky1998}, and \eqref{mainthm_HgalK} and \eqref{mainthm_Rep} being the usual genuine $\Z/2$-equivariant stable homotopy category and $\Z/2$-representation rings, respectively. Section \ref{Gk-realization} contains the main result. As described above, the basic idea of the proof is that the $\bbb{A}^1$-Euler characteristic of an elliptic curve and the alternating sum of its \'etale cohomology groups are incompatible as Galois representations.
 
\section{Enriched Euler characteristics and genuine $\Gal_{\bbb{R}}$-realization}\label{GR-realization}

Let $G$ be a finite group. Let $\A(G)$ denote the Burnside ring, defined as the group completion (or Grothendieck group) of the semi-ring of isomorphism classes of finite $G$-sets with addition and multiplication given by disjoint union and product respectively. 

The Burnside ring is the endomorphisms of the monoidal unit in the Burnside category $\Burn(G)$, whose definition we now recall. The objects of $\Burn(G)$ are finite $G$-sets. Given finite $G$-sets $T$ and $S$, consider the set of equivalence classes of diagrams of $G$-sets $$S \leftarrow U \rightarrow T$$ under the equivalence relation $[S \leftarrow U \rightarrow T]\sim[S \leftarrow U' \rightarrow T] $ when there is a $G$-equivariant bijection $\alpha: U \to U'$ such that the diagram $$\xymatrix{ & \ar[dl] U \ar[dd]_{\cong}^{\alpha} \ar[dr]& \\ S & &T \\ & \ar[ul] U' \ar[ur]&}$$ commutes. The morphisms $\Burn(G)(S,T)$ from $S$ to $T$ are the elements of the group completion of this set of equivalence classes under the operation $$[S \leftarrow U \rightarrow T] + [S \leftarrow U' \rightarrow T] = [S \leftarrow U \coprod U' \rightarrow T].$$ Composition of morphisms $S \leftarrow U \rightarrow T$ and $T \leftarrow U' \rightarrow Q$ is given by pullback \begin{equation}\label{BurnGcomp}\xymatrix{ && \ar[dl] V \ar[dr] && \\& \ar[dr]\ar[dl]U& &U' \ar[dr]\ar[dl] &\\ S&& T&& Q}\end{equation} and the composition $$\Burn(G)(S,T)\times \Burn(G)(T,Q) \to \Burn(G)(S,Q)$$ is defined by \eqref{BurnGcomp} and bilinearity. The symmetric monoidal structure takes $G$-sets $S$ and $T$ to their product $S \times T$, so $\A(G) = \Burn(G)(\ast, \ast)$, where $\ast$ denotes the one-point $G$-set.

An object $X$ of a symmetric monoidal category with unit $\Sphere$ and monoidal product $\wedge$ is said to be fully dualizable if there is a dual object $DX$ together with coevaluation and evaluation morphisms $$\eta: \Sphere \to X \wedge DX, ~~\epsilon: DX \wedge X \to \Sphere  $$ such that the compositions $$X \stackrel{\eta \wedge 1_X}{\to}  X \wedge DX \wedge X \stackrel{1_X \wedge \epsilon}{\to} X ,$$ $$DX \stackrel{1_{DX} \wedge  \eta }{\to}  DX \wedge X \wedge DX \stackrel{ \epsilon \wedge 1_{DX}}{\to} DX $$ are the identities. For example, every element of $\Burn(G)$ is fully dualizable with $DX = X$. 

Given an endomorphism $f$ of  a fully dualizable object $X$, the trace $\Tr (f)$ is the endomorphism of $\Sphere$ determined by the composition $$ \Sphere \stackrel{\eta}{\to} X \wedge DX \stackrel{f \wedge 1_{DX} }{\to} X \wedge DX\stackrel{\tau}{\to} DX \wedge X \stackrel{\epsilon}{\to} \Sphere,$$ where $\tau$ is the map which switches the order of the two factors. The Euler characteristic $\chi(X)$ is the trace of the identity $\chi(X) = \Tr (1_X)$. The following is a straightforward consequence of the definitions and is not original, but it is included for convenience and clarity. 

\begin{lm}\label{lm:TrTUT} Let $T$ and $U$ be finite $G$-sets, and let $T \stackrel{i}{\leftarrow} U \stackrel{j}{\rightarrow} T$ be a diagram of finite $G$-sets.  For elements $t$ and $t'$ in $T$, let $U_{t,t'}$ be the set $U_{t,t'} = i^{-1}(t) \cap j^{-1}(t')$. Then there are the following equalities in the Burnside ring $\A(G)$.
\begin{enumerate}
\item\label{TrUinBurnG} $\Tr[T \leftarrow U \rightarrow T] = \coprod_{t \in T} U_{t,t}$.
\item\label{chiT=T} $\chi (T) = T$.
\end{enumerate}
\end{lm}

\begin{proof}
We first show \eqref{TrUinBurnG}. Let $\Delta: T \to T \times X$ denote the diagonal. Then $\eta$ and $\epsilon$ are the morphisms $\ast \leftarrow T \stackrel{\Delta}{\to} T \times T $ and $  T \times T\stackrel{\Delta}{\leftarrow} T \to \ast$ respectively. The composition of $[T \leftarrow U \rightarrow T] \times 1_T $ and $\tau$ is $T \times T \stackrel{i \times 1_T}{\leftarrow} U \times T \stackrel{1_T \times j}\rightarrow T \times T$. The pullback diagram $$\xymatrix{  U\ar[d]^i\ar[r]_{1_U \times i} &U \times T \ar[d]^{i \times 1_T} \\T \ar[r]_{\Delta} & T \times T} $$ shows that the composition $\eta \circ ([T \leftarrow U \rightarrow T] \times 1_T) \circ \tau$ is $\ast \leftarrow U \stackrel{i \times j}{\to} T \times T$. The composition of this last map with $\eta$ shows $\Tr[T \leftarrow U \rightarrow T] = \coprod_{t \in T} U_{t,t}$. \eqref{chiT=T} follows from \eqref{TrUinBurnG}.
\end{proof}

Let $\calH_{G}(\Spt)$ denote the homotopy category of genuine $G$-spectra \cite[XII 5]{Alaska}, $\Sphere$ denote the sphere spectrum, and $\calH_{G}(\Spt)^d$ denote the full subcategory of fully dualizable objects.

It is a result of Segal \cite{SegalICM} and tom Dieck \cite{tomDieck_Trans} that there is a ring isomorphism \begin{equation}\label{AGcongEnd}\A(G) \cong \End_{\calH_{G}(\Spt)}(\Sphere),\end{equation} so for $X$ in $\calH_{G}(\Spt)^d$ we consider $\chi(X)$ to be an element of the Burnside ring $\A(G)$. The ring homomorphism \eqref{AGcongEnd} is induced from the fully faithful, symmetric, monoidal functor $\Sigma^{\infty}_+: \Burn(G) \to \calH_{G}(\Spt)$ taking a finite $G$-set $T$ to the suspension spectrum of $T_+$ \cite[Corollary 3.2 XIX]{Alaska}. 

Let $R$ be a ring. Let $\Rep(G,R)$ denote the ring with generators given by isomorphism classes of finitely generated $R$-modules with $G$-action subject to the relation that $B = A \oplus C$ when there is a short exact sequence $$A \to B \to C.$$ Addition and multiplication in $\Rep(G,R)$ are induced from $\oplus$ and $\otimes$ respectively. For $X$ in $\calH_{G}(\Spt)$, the cohomology groups $H^i(X,R)$ determine elements of $\Rep(G,R)$. 

\begin{df}\label{assoc_def_rmk}
Let $\assoc: \A(G) \to \Rep(G,R)$ be the map determined by taking a finite $G$-set $T$ to the permutation representation on $\oplus_T R$.
\end{df}

\begin{pr}\label{assocchiX=repchi}

\begin{enumerate}
\item \label{pr:assocchiX=repchi:case1} Let $R$ be any ring. For all finite $G$-CW complexes $X$ we have that $$\assoc \chi (X) = \sum_{i=-\infty}^{\infty} (-1)^i H^i(X,R).$$ 
\item \label{pr:assocchiX=repchi:case2} Let $R$ be a field of characteristic not dividing the order of $G$. Then for all $X$ in $\calH_{G}(\Spt)^d$ $$\assoc \chi (X) = \sum_{i=-\infty}^{\infty} (-1)^i H^i(X,R),$$ and in particular the right hand side is a well-defined element of $\Rep(G,R)$ in the sense that only finitely many of the summands are non-zero and the non-zero summands are finitely generated $R$-modules.
\end{enumerate}
\end{pr}

The proof of Proposition \ref{assocchiX=repchi} \eqref{pr:assocchiX=repchi:case1} is a straightforward induction on the cells of $X$. To prove \eqref{pr:assocchiX=repchi:case2}, we will have use of the following lemmas.

\begin{lm}\label{lm:int_Tr-to-perm}
Let $G$ be a finite cyclic group. Let $R$ be a field of characteristic not dividing the order of $G$. Let $T$ be a finite $G$-set, and $V \subseteq \assoc T$ be a $G$-fixed subspace such that the corresponding representation $\rho: G \to \GL V$ satisfies the property that $\Tr(\rho(g))$ is the image of an integer under the canonical map $\mathbb{Z} \to R$ for all $g$ in $G$. Then there are $G$-sets $E$ and $I$ such that $V \oplus \assoc E \cong \assoc I$. 
\end{lm}

\begin{proof}
If $G$ has order $1$, then the lemma is true. Let $G$ have order $n>1$ and assume inductively that the lemma holds for cyclic groups of smaller orders. Because matrices which are similar over the algebraic closure of $R$ are similar over $R$, we may assume $R$ is algebraically closed. Thus $\assoc T$ and $V$ decompose into simultaneous $1$-dimensional eigenspaces whose eigenvalues are $n$th roots of unity. Let $d$ denote the largest order of an eigenvalue of $V$, and note that $d$ divides $n$. If $d=1$, then $V\cong \assoc I$ where $I$ is a set with trivial $G$-action of cardinality equal to the dimension of $V$, so the lemma is true. 

We now induct on $d$. Let $g$ denote a generator of $G$. The condition that  $\Tr(\rho(g))$ is an integer implies that the eigenspaces of $g$ in $V$ associated to the primitive $d$th roots of $1$ are all the same dimension. Call this dimension $a$.  Let $d=\prod_{j=1}^{l} p_j^{e_j}$ be the prime factorization of $d$. Let $E_j$ denote the $G$-set consisting of a single orbit of cardinality $d/p_j$, and let $E' = \coprod_{i=1}^a \coprod_{j=1}^l E_j$. Then for each $d$th root of unity, we may choose an $a$ dimensional eigenspace of $V \oplus \assoc E'$. Let $V'$ denote the direct sum of these chosen eigenspaces and let $V'' \subset V \oplus \assoc E'$ denote a complementary $G$-representation to $V'$. By construction $V' \cong \assoc (\coprod_{j=1}^a I')$, where $I'$ is the $G$-set with a single orbit of cardinality $d$.  In particular, this representation satisfies the property that the trace of every element of $G$ is an integer. It follows that the same property holds for $V''$. By induction on $d$, there are $G$-sets $E''$ and $I''$ such that $V'' \oplus \assoc E'' \cong \assoc I''$. Thus $V \oplus \assoc E' \oplus \assoc E'' \cong  \assoc(\coprod_{j=1}^a I') \oplus \assoc I''$. We may therefore let $E = E' \coprod E''$ and $I =  (\coprod_{j=1}^a I') \coprod I''$, and the lemma is true.

\end{proof}

\begin{lm}\label{assoc_det_set_G=Z/n}
Let $G$ be a finite cyclic group. Assume $R$ has a prime ideal of residue characteristic not dividing the order of $G$. If $T$ and $T'$ are two finite $G$-sets such that $\assoc T \cong \assoc T'$, then $T \cong T'$.
\end{lm}

\begin{rmk}
Lemma \ref{assoc_det_set_G=Z/n} becomes false when $G$ is $\Z/2\times \Z/2$ \cite[\S 7 Example 2]{Roberts-Equivariant_Milnor}.
\end{rmk}

\begin{proof}
By assumption, there is a ring map $R \to K$ where $K$ is an algebraically closed field of characteristic not dividing the order of $G$. We may replace $R$ by $K$, allowing us to decompose $\assoc T \cong \assoc T'$ into simultaneous $1$-dimensional eigenspaces whose eigenvalues are $n$th roots of unity. Let $d$ denote the largest order of such an eigenvalue. If $d=1$, then both $T$ and $T'$ have trivial $G$-actions and $\vert T \vert = \vert T' \vert = \dim \assoc T'$, so the lemma is true. Assume by induction that the lemma holds for all smaller values of $d$. Let $g$ be a generator of $G$ and $\zeta$ be a primitive $d$th root of unity. The number of $d$-cycles in $\assoc T$ is the dimension of the $\zeta$-eigenspace of $T$. Since the same holds for $T'$, it follows that $T$ and $T'$ contain the same number of $d$-cycles. Furthermore, removing the $d$-cycles from both $T$ and $T'$ results in finite $G$-sets with isomorphic associated permutation representations for which the lemma holds inductively. Thus $T \cong T'$.
\end{proof}

\begin{lm}\label{fBurnside_idempotentTr(f)}
Let $G$ be a finite group and let $f:T \to T$ be an idempotent in the Burnside category of $G$. Let $R$ be a field of characteristic not dividing the order of $G$. Then $\assoc \Tr(f) = \Image H^0(f)$ in $\Rep(G,R)$. 
\end{lm}

We establish a few preliminaries before proving Lemma \ref{fBurnside_idempotentTr(f)}.

\begin{rmk}\label{calculate_trImageH0(f)andTrf}
We may directly calculate the trace functions $\chi_{\assoc \Tr(f)}$ and $\chi_{\Image H^0(f)}$ of $\assoc \Tr(f)$ and $\Image H^0(f)$. Since $f$ is idempotent, $$\chi_{\Image H^0(f)}(g) = \Tr(H^0(f) \vert \Image H^0(f) )= \Tr(H^0(g)H^0(f) \vert H^0(T)).$$ Let $f = [T \stackrel{i}{\leftarrow} U \stackrel{j}{\rightarrow} T] - [T \stackrel{i}{\leftarrow} V \stackrel{j}{\rightarrow} T]$. Let $\delta_t$ be the function $\delta_t: T \to R$ such that $\delta_t(t) = 1$ and $\delta_t(t') = 0$ for $t' \neq t$. In the below, $i_*$ denotes the push forward on $H^0$ associated to $i$, and $g^*, j^*$ are the corresponding pullbacks. We compute \begin{align*}
\chi_{\Image H^0(f)}(g) = \Tr(H^0(g)H^0(f) \vert H^0(T)) = \sum_{t \in T} (g^* i_* j^* \delta_t)(t) =  \sum_{t \in T} (i_* g^*  j^* \delta_t)(t) \\
 =  \sum_{t \in T} (\sum_{u \in U : i(u) = t} (g^* j^* \delta_t)(u) - \sum_{v \in V : i(v) = t} (g^* j^* \delta_t)(v)) \\
 =  \sum_{t \in T} (\sum_{u \in U : i(u) = t} \delta_t(gju) - \sum_{v \in V : i(v) = t}  \delta_t(gjv)) \\
 = \vert \{u \in U : i(u) = gj(u) \} \vert - \vert \{v \in V : i(v) = gj(v) \} \vert .
\end{align*}

It follows from Lemma \ref{lm:TrTUT} that \begin{align*} \chi_{\assoc \Tr(f)}(g) = \vert \{u \in U : gu=u, i(u) = j(u) \} \vert - \vert \{v \in V : gv =v, i(v) = j(v) \} \vert.\end{align*}

It is not clear from these calculations that $\chi_{\Image H^0(f)} = \chi_{\assoc \Tr(f)}$. We will show in Lemma \ref{fBurnside_idempotentTr(f)} that they are.

It is clear that $\chi_{\Image H^0(f)}$ takes values in the integers, and it follows from $$\chi_{\Image H^0(f)} + \chi_{\Ker H^0(f)} = \chi_{H^0(T)},$$ that $\chi_{\Ker H^0(f)}$ also takes values in the integers. 
\end{rmk}

To prove Lemma \ref{fBurnside_idempotentTr(f)}, we will make use of the following variation of the Burnside category.

\begin{df}
Define the category $\Burn(G,R)$ to have objects $G$-sets, and for $G$-sets $S$ and $T$, the morphisms $\Burn(G,R)(S, T)$ are $R$-linear combinations of equivalence classes of diagrams of $G$-sets $S \leftarrow U \rightarrow T$. Composition is defined to be $R$-bilinear and induced from \eqref{BurnGcomp} as in the definition of the Burnside category. 
\end{df}

$\Burn(G,R)(S, T)$ is a symmetric, monoidal category with the monoidal structure defined by the product of $G$-sets, and $\Tr[T \leftarrow U \rightarrow T] = \coprod_{t \in T} U_{t,t}$ by the proof of Lemma \ref{lm:TrTUT}. 

\begin{lm}\label{TrAB=TrBA}
Let $A$ and $B$ be morphisms in $\Burn(G,R)(T, T)$. Then $\Tr(A \circ B)= \Tr (B\circ A)$.
\end{lm}

\begin{proof} Since composition of morphisms is $R$-bilinear, and $\Tr$ is $R$-linear, it suffices to prove the claim where $A$ and $B$ are determined by diagrams of $G$-sets $T \leftarrow A \rightarrow T$ and $T \leftarrow B \rightarrow T$, respectively. Then:
\begin{align*} \Tr(A \circ B)& \cong \coprod_{t \in T} (A \circ B)_{tt} \\
&\cong  \coprod_{t \in T} \coprod_{t' \in T} A_{t, t'} \times B_{t',t} \\
& \cong  \coprod_{t , t' \in T} A_{t, t'} \times B_{t',t} \\
& \cong \coprod_{t , t' \in T} B_{t',t} \times A_{t, t'} \\
&\cong  \coprod_{t' \in T} \coprod_{t \in T} B_{t',t} \times A_{t, t'} \\ 
&\cong \coprod_{t' \in T} (A \circ B)_{t',t'}  \cong  \Tr(B \circ A) \end{align*}
\end{proof}

\begin{proof} (of Lemma \ref{fBurnside_idempotentTr(f)})
By Maschke's theorem \cite[Theorem 3.1, Theorem 3.5]{Etingofetal_Intro_Rep_thy}, it suffices to show that the trace functions $\chi_{\assoc \Tr(f)}$ and $\chi_{\Image H^0(f)}$ of $\assoc \Tr(f)$ and $\Image H^0(f)$ are the same. To do this, we may show for each $g$ in $G$ that $\chi_{\assoc \Tr(f)}(g) = \chi_{\Image H^0(f)}(g)$, thus reducing to the case where $G$ is a finite cyclic group. 

By Remark \ref{fBurnside_idempotentTr(f)}, $\chi_{\Image H^0(f)}$ takes values in the integers, so we may apply Lemma \ref{lm:int_Tr-to-perm} to $\Image H^0(f) \subseteq \assoc T$. Thus there are $G$-sets $E$ and $I$ such that $\Image H^0(f) \oplus \assoc E \cong \assoc I$. Let $f': T \coprod E \to T \coprod E$ be the idempotent in the Burnside category of $G$ defined $f'=f \coprod 1_E$. Applying Remark \ref{fBurnside_idempotentTr(f)} and Lemma \ref{lm:int_Tr-to-perm} to $\Ker H^0(f') \subseteq \assoc (T \coprod)$, there are $G$-sets $E'$ and $I'$ such that $\Ker H^0(f') \oplus \assoc E' \cong \assoc I'$. Define $f'': T \coprod E \coprod E' \to T \coprod E  \coprod E' $ be the idempotent in the Burnside category of $G$ defined $f''=f' \coprod 0_{E'}$. 

By Lemma \ref{lm:TrTUT}, $\Tr f'' \cong E \coprod \Tr f$, and $\Image H^0(f'') \cong \Image H^0(f) \oplus \assoc E$. Therefore it suffices to prove the lemma with $f$ replaced by $f''$.

Since $f''$ is an idempotent, $$\assoc ( T \coprod E \coprod E' \to T) \cong \Ker H^0(f'') \oplus  \Image H^0(f'').$$ By construction, $\Image H^0(f'') \cong \Image H^0(f') \cong \Image H^0(f) \oplus \assoc E \cong I$, and $\Ker H^0(f'') \cong \Ker H^0(f') \oplus \assoc E' \cong \assoc I'$. Replacing $T$ by $T \coprod E \coprod E'$ and $f$ by $f''$, we may therefore assume that there exist $G$-sets $I$ and $I'$ such that $\Image H^0(f) \cong \assoc I$ and $\Ker H^0(f) \cong \assoc I'$.

Since $f$ is an idempotent, there is an isomorphism $\assoc T \cong \Image H^0(f) \oplus \Ker H^0(f)$, whence an isomorphism $\eta: \assoc T \to \assoc (I \coprod I')$. Define $T'= I \coprod I'$. For $t \in T$ and $t' \in T'$, let $\eta_{t' ,t}$ denote the corresponding entry of the matrix of $\eta$ with respect to the bases $T$ and $T'$, and similarly define $(\eta^{-1})_{t,t'}$ to be the matrix entry of the inverse of $\eta$. Note that $\eta_{t' ,t} (t',t)$ is an element of $\assoc (T' \times T)$. The subset $\coprod_{t \in T,t' \in T'} \eta_{t' ,t} (t',t)$ of $\assoc (T' \times T)$ is invariant under the action of $G$ because $\eta$ is a $G$-isomorphism. It follows that we may view $\coprod_{t \in T,t' \in T'} \eta_{t' ,t} (t',t)$ as an $R$-linear combination of $G$-subsets of $T' \times T$. The analogous assertion holds for $\eta'$. 

The $R$-linear combination of $G$-invariant subsets of $T' \times T$ associated to $\coprod_{t \in T,t' \in T'} \eta_{t' ,t} (t',t)$ determines a morphism in $\Burn(G,R)(T, T')$, which we will denote by $\Burn(\eta)$. Similarly define $\Burn(\eta^{-1})$ in $\Burn(G,R)(T', T)$. By construction, $\Burn(\eta) \circ f  \circ \Burn(\eta^{-1}) = I$ where the two maps to $T'$ are both the canonical inclusion $I \to T'$.

We claim that Lemma \ref{TrAB=TrBA} implies that \begin{align*}\Tr (\Burn(\eta) \circ f  \circ \Burn(\eta^{-1})) = \Tr (f) \end{align*} To see that we may apply Lemma \ref{TrAB=TrBA} in this case, note that by Lemma \ref{assoc_det_set_G=Z/n}, we may fix an isomorphism of $G$-sets $T \cong T'$, and consider $\Burn(\eta)$ and $\Burn(\eta^{-1})$ to both be elements of $\Burn(G,R)(T, T)$ allowing us to conclude $\Tr (\Burn(\eta) \circ f  \circ \Burn(\eta^{-1}))= \Tr(\Burn(\eta^{-1})\circ \Burn(\eta) \circ f)$. Thus $\Tr (f) = I$. Since $\assoc I \cong \Image H^0(f)$, we have shown the lemma.
\end{proof}

The following corollary of Lemma \ref{fBurnside_idempotentTr(f)} is not needed for the rest of the article, but is included as a curiosity. To put this in context, we remark that there are many interesting idempotents of the Burnside ring of a finite group $G$. The idempotents of $\A(G) \otimes \Q$ were computed by \cite{Solomon}, \cite{Yoshida}, and \cite{Gluck-idempotent}. The integral idempotents, i.e., the idempotents of $\A(G)$, were computed by Dress \cite{Dress-solvable}, see \cite[Theorem 3.3.7, Corollary 3.3.9]{Bouc-Burnside_rings}. 

\begin{co}\label{co:associdempAg}
Let $G$ be a finite group and let $R$ be a field of characteristic not dividing the order of $G$. Suppose $f$ is an idempotent in the Burnside ring $\A(G)$. Then $\assoc f$ is either $0$ or the trivial representation in $\Rep(G,R)$. 
\end{co}

\begin{proof}
Since $\Image H^0(f)$ is a submodule of $H^0(\ast)$, we have that $\Image H^0(f)$ is either $0$ or $R$ with the trivial action. The corollary then follows from Lemma \ref{fBurnside_idempotentTr(f)}.
\end{proof}

We wish to thank Serge Bouc and Alexander Duncan for useful correspondence concerning Corollary \ref{co:associdempAg}. In particular, they explicitly computed an interesting (integral) idempotent of the Burnside ring of the alternating group $G=A_5$ such that the associated element of $\Rep(G, \Q)$ is $0$. 

\begin{proof} (of Proposition \ref{assocchiX=repchi}) We first make an observation useful for proving both \eqref{pr:assocchiX=repchi:case1} and \eqref{pr:assocchiX=repchi:case2}: For a cofiber sequence $X \to Y \to Z$ in $\calH_{G}(\Spt)^d$, there is an equality $\chi(X) + \chi(Z) = \chi(Y)$ in $\A(G)$ (see for example  \cite[XVII Theorem 1.6]{Alaska}), and therefore an equality $\assoc \chi(X) + \assoc \chi(Z) = \assoc \chi(Y)$. 

Since a short exact sequence $A \to B \to C$ of $G$-modules induces the relation $A + C = B$ in $\Rep(G,R)$, it follows by induction that an exact sequence $$ 0 \to M_{n} \to M_{n+1} \to \ldots \to M_{m-1} \to M_m \to 0$$ of $G$-modules induces the relation $\sum_{i=-\infty}^{\infty} (-1)^i M_i = 0$. The cofiber sequence  $X \to Y \to Z$ induces the long exact sequence $$ \ldots \to  H^i(Z,R) \to H^i(Y,R) \to H^i(X,R) \to H^{i+1}(Z,R) \to \ldots $$ in cohomology, whence we have that $$ \sum_{i=-\infty}^{\infty} (-1)^i H^i(X,R) +\sum_{i=-\infty}^{\infty} (-1)^i H^i(Z,R)=\sum_{i=-\infty}^{\infty} (-1)^i H^i(Y,R)$$ under the hypothesis that for two (and thus all three) of $X$, $Y$, and $Z$, only finitely many terms in each sum are non-zero. Therefore if the claim holds for any two of $X$, $Y$, and $Z$ in a cofiber sequence $X \to Y \to Z,$ it holds for the third. 

We now prove \eqref{pr:assocchiX=repchi:case1}: By induction on the number of cells of the finite $G$-CW complex $X$, it therefore suffices to show the claim for $X$ a finite $G$-set. This then follows from Lemma~\ref{lm:TrTUT}~\eqref{chiT=T}.

We now prove \eqref{pr:assocchiX=repchi:case2}: Let $A$ be a $G$-CW complex which is a retract in the homotopy category of a finite $G$-CW complex $Y$. Let $\iota$ be the endomorphism of $Y$ given by the composition of the retract and the inclusion $\iota:Y \to A \to Y$.  It is formal that $\Tr(\iota) = \chi(A)$. We likewise have that the image $\Image H^i(\iota)$ of $H^i(\iota)$ is a direct summand of $H^i(Y,R)$ isomorphic to $H^i(A,R)$. We show the claim for $A$ by proving that for a $G$-equivariant map $f:Y \to Y$ which is idempotent in the homotopy category, we have \begin{equation}\label{Trf=sumimage}\assoc \Tr(f) = \sum_{i=-\infty}^{\infty} (-1)^i [\Image H^i(f)],\end{equation} and applying \eqref{Trf=sumimage} for $f=\iota$. Both sides of are unchanged under equivariant homotopy, and it follows that we may assume that $f$ is cellular \cite[I. Theorem 3.4]{Alaska}. This allows us to induct on the dimension of the top cells of $Y$, reducing to the case where $Y = T_+ \wedge S^n$ for a $G$-set $T$. Since replacing $f$ by its suspension multiplies both sides of \eqref{Trf=sumimage} by $-1$, we replace $f$ by $f \wedge S^{-n}$, reducing to the case where $Y= T_+$ and the map $f$ is now is the stable homotopy category. Then as remarked above $f$ corresponds to an idempotent $f: T \to T$ in $\Burn(G)$ \cite[Corollary 3.2 XIX]{Alaska}. Thus $ \sum_{i=-\infty}^{\infty} (-1)^i [\Image H^i(f)] = [\Image H^0(f)] $. The equality \eqref{Trf=sumimage} then follows by Lemma \ref{fBurnside_idempotentTr(f)}. 

Let $V$ be an $n$-dimensional representation of $G$ and let $S^{-V}$ be the dual of the one-point compactification $S^V$ of $V$. If $X$ is a strongly dualizable $G$-spectra, then so is $DX$ and $\chi(X) = \chi(DX)$. Thus $\chi(S^{-V})=\chi(S^V)$. Since $S^V$ is a finite $G$-CW complex, we have that $$\assoc \chi(S^V) = \sum_{i=-\infty}^{\infty} (-1)^i H^i(S^V,R) = R + (-1)^n R,$$ where the action of $G$ on the first summand is trivial, and the action of $G$ on the second summand is via the sign on the determinant, i.e. $g \in G$ acts by $\pm 1$ depending on if $g$ preserves or reverses the orientation of $S^V$. Combining with the previous, we have $$\assoc \chi(S^{-V}) = R + (-1)^n R.$$ By definition, $H^i (DX, R) = \pi_{-i} F(DX, R) = \pi_{-i}(X \wedge R) = H_{-i} (X,R)$. Thus $$H^i(S^{-V}, R) = \begin{cases} R &\mbox{if } i= 0 \\ 
R & \mbox{if } i = -n \\ 
0&\mbox{otherwise }\end{cases},$$ where for $i=0$, the action is trivial, and for $i=-n$, $G$ acts on $R$ by the sign of the determinant. Thus the claim holds for $S^{-V}$.

By \cite[Proposition 2.1]{Fausk_Lewis_May}, $X$ in $\calH_{G}(\Spt)^d$ is equivalent to $\Sigma^{-V} \Sigma^{\infty} A$ where $A$ is a finitely dominated based $G$-CW complex and $V$ is a representation of $G$. Since $A$ is finitely dominated, by definition we have a finite $G$-CW complex $Y$ such that $A$ is a retract of $Y$ in the homotopy category of $G$-CW complexes. By the above, the claim holds for $A$ and $S^{-V}$. Since the smash product of two endomorphisms of $\Sphere$ induces the multiplication of $\A(G)$, we have $\chi(\Sigma^{-V} \Sigma^{\infty} A) = \chi(S^{-V}) \chi(A)$. Thus $\assoc \chi(\Sigma^{-V} \Sigma^{\infty} A) = \assoc \chi(S^{-V}) \assoc \chi(A)$. By the K\"unneth spectral sequence, we have the equality $$\sum_{i=-\infty}^{\infty} (-1)^i H^i(\Sigma^{-V} \Sigma^{\infty} A,R) = (\sum_{i=-\infty}^{\infty} (-1)^i H^i(S^{-V},R))(\sum_{i=-\infty}^{\infty} (-1)^i H^i(A,R))$$ in $\Rep(G,R)$. The claim for $X$ then follows from the claim for $S^{-V}$ and $A$.

\end{proof}

This concludes the facts we need about genuine $G$-spectra. We now turn to the relationship between $\bbb{A}^1$-homotopy theory over $\bbb{R}$ and genuine $\Gal_{\bbb{R}}$-spectra. 

Let $\calH_{\bbb{A}^1}(\spaces{\bbb{R}})$ denote the $\bbb{A}^1$-homotopy category of simplicial presheaves on smooth schemes over $\bbb{R}$ in the sense of Morel-Voevodsky. Let $ \calH_{\Gal_{\bbb{R}}}(\sSet)$ denote the homotopy category of genuine $G$-spaces for $G = \Gal_{\bbb{R}}$. There is a symmetric monoidal functor $$LB:\calH_{\bbb{A}^1}(\spaces{\bbb{R}}) \to \calH_{\Gal_{\bbb{R}}}(\sSet),$$ called Betti realization, which takes a smooth scheme $X$ over $\bbb{R}$ to the complex points $X(\bbb{C})$ with the $\Gal_{\bbb{R}}$-action induced from the tautological action of $\Gal_{\bbb{R}}$ on $\bbb{C}$. See \cite[3, 3.3]{morelvoevodsky1998}. 

$LB$ furthermore determines a functor after stabilization. Namely, let $\calH_{\bbb{A}^1}(\Spt^{\bbb{P}^1}(\bbb{R}))$ denote the $\bbb{A}^1$-homotopy category of $\bbb{P}^1$-spectra over $\bbb{R}$. As above, let $ \calH_{\Gal_{\bbb{R}}}(\Spt)$ denote the homotopy category of genuine $G$-spectra for $G = \Gal_{\bbb{R}}$. There is a commutative diagram of Betti-realization functors \cite[4.4]{HO_Galois} $$\xymatrix{\calH_{\bbb{A}^1}(\spaces{\bbb{R}}) \ar[rr]^{LB} \ar[d]^{(\Sigma^{\infty}_{\bbb{P}^1})_+}&& \ar[d]^{(\Sigma^{\infty}_{S^{\bbb{C}}})_+}\calH_{\Gal_{\bbb{R}}}(\sSet)\\ \calH_{\bbb{A}^1}(\Spt^{\bbb{P}^1}(\bbb{R})) \ar[rr]^{LB}&& \calH_{\Gal_{\bbb{R}}}(\Spt).}$$

For $k$ a field, Morel has shown that the endomorphisms of the sphere $ \End_{\calH_{\bbb{A}^1}(\Spt^{\bbb{P}^1}(k))}(\Sphere)$ are isomorphic to the Grothendieck-Witt group $\GW(k)$ of $k$ \cite[Corollary 1.24]{morel} (see \cite[footnote p~2]{Hoyois_lef} about the non-perfect case), defined to be the group completion of the semi-ring of symmetric bilinear forms.  $\GW(k)$ is generated by $\langle a \rangle$ for $a \in k^*/(k^*)^2$ and has relations given by $\langle u \rangle + \langle -u \rangle = \langle 1 \rangle + \langle -1 \rangle$ and $\langle u \rangle + \langle v \rangle = \langle u + v\rangle + \langle (u+v)uv \rangle$ for $u,v\in k^*$ and $u + v \neq 0$. See \cite[Lemma 3.9]{morel} who cites \cite{milnor73}. The element $\langle a \rangle$ corresponds to the bilinear form $(x,y) \mapsto a xy$. We will compare $\GW(\bbb{R})$ and $\A(\Gal_{\bbb{R}})$ and for this, we need the following well-known constructions.

For a separable field extension $k \subseteq L$, there is a map $\Tr_{L/k}: \GW(L) \to \GW(k)$ defined by sending a symmetric bilinear form $b: V \times V \to L$ to the composition $\Tr_{L/k} b : V \times V \to k$ of the form $b$ with the field trace map $\Tr_{L/k}: L \to k$. In $\Tr_{L/k} b$, the vector space $V$ is now viewed as a vector space over $k$. 

The map sending a bilinear form to the ordered pair of its rank and signature determines an isomorphism $$\xymatrix{\GW(\bbb{R}) \ar[rrr]_{\cong}^{\text{rank} \times \text{signature}}&&& \Z \times \Z},$$ as can be checked using the generators and relations described above and Sylvester's law of inertia.

Since $LB$ is a functor, it determines a map $$LB: \GW(\bbb{R}) \cong \End_{\calH_{\bbb{A}^1}(\Spt^{\bbb{P}^1}(\bbb{R}))}(\Sphere) \to \End_{\calH_{\Gal_{\bbb{R}}}(\Spt)}(\Sphere) \cong \A(\Gal_{\bbb{R}}).$$  Since $LB$ is symmetric monoidal, $LB(\chi Y) = \chi(LB Y)$ for every strongly dualizable object $Y$ of $\calH_{\bbb{A}^1}(\Spt^{\bbb{P}^1}(\bbb{R}))$. It is a result of Hoyois \cite[Theorem 1.9]{Hoyois_lef} that for a separable field extension $k \subseteq L$, the fully dualizable spectrum $\Sigma^{\infty}_{\bbb{P}^1} \Spec L_+$ has Euler characteristic $\chi \Spec L = \Tr_{L/k} \langle 1 \rangle$. Thus in  $\End_{\calH_{\bbb{A}^1}(\Spt^{\bbb{P}^1}(\bbb{R}))}(\Sphere)$, we have that $\chi(\Sigma^{\infty}_{\bbb{P}^1} \Spec \C_+) = \langle 1 \rangle + \langle -1 \rangle$. Since $LB (\Sigma^{\infty}_{\bbb{P}^1} \Spec \C_+) $ is the element of $\calH_{\Gal_{\bbb{R}}}(\Spt)$ corresponding to the finite $ \Gal_{\bbb{R}}$-set given by $\Gal_{\bbb{R}}$ with left translation, it follows that the map $$ \GW(\bbb{R}) \to \A(\Gal_{\bbb{R}})$$ is determined by $$\xymatrix{  \GW(\bbb{R}) \ar[d]_{\text{rank} \times \text{signature}} \ar[r] &  \A(\Gal_{\bbb{R}}) \ar[d]^{(T \mapsto \vert T \vert) \times (T \mapsto \vert T^{\Gal_{\bbb{R}} }\vert)} \\ \bbb{Z}^2 \ar[r]^1 & \bbb{Z}^2},$$ and in particular is an isomorphism. Here the notation $\vert T \vert$ for a finite set $T$ denotes the cardinality of $T$. We remark that there is much more to say about the relationship between the Burnside ring and the Grothendieck-Witt group. See, for example, the recent work of Kyle Ormsby and Jeremiah Heller \cite{HO_Galois}. 

The following proposition about an elliptic curve is one consequence of the existence of a realization over $k=\mathbb{R}$.  In the next section, we show that the analogous result often fails to hold when $k$ is a more complicated field and then use this fact to show that a suitable realization functor cannot exist.

\begin{pr}\label{chirER=0}
Let $E$ be an elliptic curve over $\R$. Then \begin{enumerate}
\item \label{chiRepEC=0} For any ring $R$, we have $\sum_{i=0}^{2}(-1)^i H^i (E(\bbb{C}),R) = 0$ in $\Rep(\Gal_{\bbb{R}},R).$ 
\item \label{chiRepetEC=0} For a finite ring $R$, we have $\sum_{i=0}^{2}(-1)^i H^i_{\et}(E_{\bbb{C}},R) = 0$ in $\Rep(\Gal_{\bbb{R}},R).$ 
\end{enumerate}
\end{pr}

\begin{proof}
$E$ is strongly dualizable in $\calH_{\bbb{A}^1}(\spaces{\bbb{R}})$, and by \cite[Theorem 1.5]{Hoyois_lef}, the Euler characteristic of $E$ is $0$. Since $LB$ is a symmetric monoidal functor, it follows that  $LB E$ is strongly dualizable in $\calH_{\Gal_{\bbb{R}}}(\Spt)$ and that $\chi LB E = 0$.

\noindent \eqref{chiRepEC=0}: By \cite{Illman}, $LB E$ is a finite $\Gal_{\bbb{R}}$-CW complex, so we may apply Proposition~\ref{assocchiX=repchi} \eqref{pr:assocchiX=repchi:case1} to $LB E$. It follows that $ \sum_{i=-\infty}^{\infty} (-1)^i H^i(LB E,R) = 0$. Since $LB E = E(\C)$ and $H^i(E(\C), R) = 0$ for $i<0$ and $i>2$, we have $\sum_{i=0}^{2}(-1)^i H^i (E(\bbb{C}),R) = 0$ as claimed.

\noindent \eqref{chiRepetEC=0}: For $R$ finite, we have a natural isomorphism $H^i (E(\bbb{C}),R) \cong H^i_{\et}(E_{\bbb{C}},R) $ by \cite[III Theorem 3.12]{MilneECbook} \cite[XI]{sga4III}, so \eqref{chiRepetEC=0} follows from \eqref{chiRepEC=0}.
\end{proof}

\begin{rmk} We can also verify Proposition \eqref{chirER=0} directly. Indeed, when $R=\mathbb{C}$, the representations $H^{i}(E(\mathbb{C}), R)$ can be described explicitly as follows.
Let $\C$ denote the trivial representation and $\C(1)$ denote the sign representation in $\Rep(\Gal_{\bbb{R}},\C)$. Then in $\Rep(\Gal_{\bbb{R}},\C)$, we have equalities $H^0(E(\bbb{C}),\C) = \C$, $H^2(E(\bbb{C}),\C) = \C(1)$, $H^1(E(\bbb{C}),\C) = \C + \C(1)$. Indeed, the first equality follows because $\Gal_{\R}$ acts trivially on the single connected component of $E(\C)$; the second equality follows because the non-trivial element of $\Gal_{\R}$ reverses orientation; the third equality follows because the cup product gives an isomorphism of $\Gal_{\R}$-representations $\wedge^2 H^1(E(\bbb{C}),\C) \cong H^2(E(\bbb{C}),\C)$, and in the representation ring, $H^1(E(\bbb{C}),\C)$ is a direct sum of one-dimensional representations.
\end{rmk}

\section{Enriched Euler characteristics and restrictions on genuine $\Gal_{k}$-realization}\label{Gk-realization}

Now let $G$ be a profinite group, or more specifically a Galois group. One could hope to construct a homotopy category $\calH_{G}(\Spt)$ of genuine $G$-spectra or $G$-pro-spectra, and a Burnside ring $\A(G) = \End_{\calH_{G}(\Spt)}(\Sphere)$. For example, homotopy theories of $G$-spectra are constructed in \cite{Fausk} \cite{Quick-Profinite_G-spectra} and \cite{BarwickI}. We have in mind that there is some appropriate sort of $G$-set and corresponding Burnside category such that a suspension spectrum functor is fully faithful into $\calH_{G}(\Spt)$, and that objects in $\calH_{G}(\Spt)$ can be constructed from colimits of suspensions of these $G$-sets, as is the case when $G$ is finite. The point here being that the Euler characteristic of a $G$-set would be itself, and spaces would be built from $G$-sets, giving credibility to an analogue of Proposition \ref{assocchiX=repchi}. For any profinite group, a category of genuine $G$-equivariant spectra has been constructed from a Burnside category by Barwick in \cite{BarwickI}, and studied by Barwick, Glasman, and Shah in \cite{BGSII}.

Let $k$ be a field, and let $\kbar$ denote an algebraic closure of $k$ and $\Gal_k = \Gal(\kbar/k)$. One could furthermore hope to construct an  {\em \'etale realization functor} $$L \Et: \calH_{\bbb{A}^1}(\Spt^{\bbb{P}^1}(k)) \to \calH_{\Gal_k}(\Spt)$$ from the stable $\bbb{A}^1$-homotopy category of $\bbb{P}^1$-spectra over $k$ to genuine $\Gal_k$-spectra or pro-spectra, appropriately completed away from the characteristic of $k$. For example, Quick has constructed an \'etale realization functor from the stable $\bbb{A}^1$-homotopy category to a stable homotopy category of profinite spaces in \cite{Quick-stable_realization}.  \'Etale realization functors have been constructed and studied by Ayoub \cite{Ayoub-realization_etale} and \cite[Section~7.2]{CD-etale_motives} in a generalization of the following context. The category of Voevodsky motives is analogous to the stable $\bbb{A}^1$-category \cite[Section 2]{AyoubICM}.  Let $\Lambda$ be the ring $\Lambda= \bbb{Z}/\ell$ for a prime $\ell$ different from the characteristic of $k$, and assume that $k$ is perfect. Let $D(\Sheaves(k_{\et}, \Lambda))$ denote the derived category of sheaves of $\Lambda$-modules on the small \'etale site of $\Spec k$. There is an \'etale realization functor from Voevodsky motives to $D(\Sheaves(k_{\et}, \Lambda))$. Since a sheaf of $\Lambda$-modules on the small \'etale site of $\Spec k$ is a $\Lambda$-module with an action of $\Gal_k$, the derived category $D(\Sheaves(k_{\et}, \Lambda))$ is similar to spectra (of $H \Lambda$-modules) equipped with an action of $\Gal_k$. We wish to draw a similarity between $D(\Sheaves(k_{\et}, \Lambda))$ and a homotopy theory of spectra with a $\Gal_k$-action and contrast $D(\Sheaves(k_{\et}, \Lambda))$ with a notion of genuine $\Gal_k$-spectra. For example, the endomorphisms of the symmetric monoidal unit of $D(\Sheaves(k_{\et}, \Lambda))$ is $\Lambda$, in contrast to the Burnside ring. Since \'etale realization functors exist in powerful contexts, one could hope for an \'etale realization functor to genuine Galois equivariant spectra or pro-spectra.

We have in mind that applying the \'etale realization functor $L \Et$ to the suspension spectrum of a smooth scheme gives an appropriate suspension spectrum of the \'etale topological type of Artin-Mazur \cite{Artin-Mazur} and Friedlander \cite{Friedlander}. In other words, we have in mind that applying $L \Et$ is compatible with an unstable \'etale realization functor $L \Et$ as constructed by Isaksen \cite{Isaksen-etale_realization} in the non-equivariant context. From this, we are lead to the following two expectations.
\begin{itemize}
\item Since the \'etale topological type $\Et ((X \times Y)_{\kbar})$ is equivalent to the product $\Et (X_{\kbar}) \times \Et (Y_{\kbar})$, cf. \cite[Corollarie~1.11]{sga4andhalf}, it is reasonable to hope that $L \Et$ is symmetric monoidal.  
\item Since the \'etale cohomology of a smooth scheme $X$ with finite coefficients $R$ is isomorphic to the cohomology of the \'etale topological type $\Et X$ with coefficients in $R$,  this would result in an isomorphism of $\Gal_k$-representations $$H^i_{\et}(X_{\kbar}, R) \cong H^i(L \Et X_{\kbar}, R).$$ This implies that if the Euler characteristic of a genuine $G$-spectrum or pro-spectrum is connected to the $\Gal_k$-representations given by its cohomology groups as in Proposition \ref{assocchiX=repchi}, then applying a homomorphism $\assoc$ from $\A(G)$ to a representation ring sends this Euler characteristic to $\sum_{i=-\infty}^{\infty} H^i_{\et}(X_{\kbar}, R).$ 
\end{itemize}

The purpose of this paper is to show that, for many fields $k$, it is impossible to simultaneously satisfy these hopes, i.e., such an \'etale realization functor does not exist. 

\begin{tm}\label{tm:nonexistence}
Let $k$ be a global field, local field with infinite Galois group, or finite field. Let $p$ be a prime different from the characteristic of $k$. It is impossible to simultaneously construct all of the following \begin{enumerate}
\item A symmetric monoidal category $\calH_{\Gal_k}(\Spt)$, enriched over abelian groups. Let $\A(\Gal_k)= \End_{\calH_{\Gal_k}(\Spt)}(\Sphere)$ denote the endomorphisms of the symmetric monoidal unit $\Sphere$. 

\item For all $n$, a representation ring $\Rep( \Gal_k, \Z/{p^n})$ with coefficients in $\Z/{p^n}$. 

\item \label{thm:nonexistence:hyp:realization}  A symmetric monoidal additive functor $L \Et: \calH_{\bbb{A}^1}(\Spt^{\bbb{P}^1}(k)) \to \calH_{\Gal_k}(\Spt)$, and ring homomorphisms $\assoc_n: \A(G) \to \Rep(\Gal_k, \Z/p^n)$ such that $$\assoc_n \chi (L\Et \Sigma^{\infty}_{\bbb{P}^1}X_+) = \sum_{i=-\infty}^{\infty} (-1)^i [H^i_{\et} (X_{\kbar}, \Z/p^n)]$$ for smooth proper schemes $X$ over $k$.%, or even just for $X$ an elliptic curve over $k$.
\end{enumerate}

\end{tm}

\begin{lm}\label{chirepEFellneq0pn}
Let $E$ be an elliptic curve over a finite field $\bbb{F}_{\ell}$, and let $p$ be any prime not dividing $\ell$. Suppose that for all $n$, we have a ring $\Rep( \Gal_k, \Z/{p^n})$ as in Theorem \ref{tm:nonexistence}. There exists a positive integer $n$, such that $\sum_{i=-\infty}^{\infty} (-1)^i [H^i_{\et} (E_{\kbar}, \Z/p^n)]$ is non-zero in $\Rep(\Gal_k,\bbb{Z}/p^n)$. 
\end{lm}

\begin{proof}
Suppose the contrary. Then for all $g$ in $\Gal_k$ and all $n$, we have the equality $$\Tr g \vert H^1(E_{\kbar}, \Z/p^n) = \Tr g \vert H^0(E_{\kbar}, \Z/p^n) + \Tr g \vert H^2(E_{\kbar}, \Z/p^n)$$ in $\Z/p^n$. By definition, $H^i(E_{\kbar}, \Q_p) = \Q_p \otimes \varprojlim_n H^i(E_{\kbar}, \Z/p^n)$ \cite[V I]{MilneECbook}. Therefore \begin{equation}\label{TrgHiEQp}\Tr g \vert H^1(E_{\kbar}, \Q_p) = \Tr g \vert H^0(E_{\kbar}, \Q_p) + \Tr g \vert H^2(E_{\kbar}, \Q_p)\end{equation} in $\Q_p$. Let $F$ be the geometric Frobenius in $\Gal_k$, i.e., the inverse of the arithmetic Frobenius $a \mapsto a^{\ell}$ for all $a$ in $\kbar$. By the Weil Conjectures, $\Tr F \vert H^1(E_{\kbar}, \Q_p) = a_1 + a_2$, where $a_i$ are algebraic integers with absolute value $\ell^{1/2}$  \cite[Theorem~IV.1.2]{Freitag_Kiehl}. Since $\Gal_k$ acts trivially on $H^0(E_{\kbar}, \Q_p)$ and there is an isomorphism $H^2(E_{\kbar}, \Q_p) \cong \Q_p(1)$, Equation~\eqref{TrgHiEQp} for $g=F$ becomes $$ a_1 + a_2 = 1+\ell.$$ Taking the absolute value of both sides, we have $$ 2 \ell^{1/2}=  \vert a_1 \vert + \vert a_2 \vert \geq \vert a_1 + a_2 \vert = 1 + \ell  ,$$ which is impossible for any prime power $\ell$.
\end{proof}

\begin{proof}
(of Theorem \ref{tm:nonexistence}) Suppose $k$ is a finite field $k=\bbb{F}_{\ell}$. Choose an elliptic curve over $k=\bbb{F}_{\ell}$. To see that this is indeed possible, note that the Weierstra\ss  ~equation $y^2 = x (x-1)(x+1)$ determines an elliptic curve when $\bbb{F}_{\ell}$ has odd characteristic. When $\ell=2^d$, the equation $y^2+y=x^3$ produces the desired elliptic curve. By Lemma \ref{chirepEFellneq0pn}, there exists a positive integer $n$  such that $\sum_{i=-\infty}^{\infty} (-1)^i [H^i_{\et} (E_{\kbar}, \Z/p^n)]$ is non-zero in $\Rep(\Gal_k,\bbb{Z}/p^n)$. By hypothesis \eqref{thm:nonexistence:hyp:realization} of the theorem, $\assoc_n \chi (L\Et \Sigma^{\infty}_{\bbb{P}^1}E_+) = \sum_{i=-\infty}^{\infty} (-1)^i [H^i_{\et} (E_{\kbar}, \Z/p^n)]$. It follows that $\chi (L\Et \Sigma^{\infty}_{\bbb{P}^1}E_+) $ is non-zero. Since $L \Et$ is a symmetric monoidal functor (by \eqref{thm:nonexistence:hyp:realization}), it follows that $L\Et \chi (\Sigma^{\infty}_{\bbb{P}^1}E_+) = \chi (L\Et \Sigma^{\infty}_{\bbb{P}^1}E_+)$, and therefore $\chi (\Sigma^{\infty}_{\bbb{P}^1}E_+)$ is non-zero. This contradicts \cite[Theorem 1.5]{Hoyois_lef}.

Suppose $k$ is a local field with infinite Galois group, i.e., $k$ is a finite extension of $\Q_{\ell}$ or $\bbb{F}_{\ell}((t))$, and let $\mathcal{O}_{k}$ denote the integral closure of $\Z_{\ell}$ or $\bbb{F}_{\ell}[[t]]$ in $k$. We may choose an elliptic curve $E$ over $k$ with good reduction by choosing an elliptic curve over the residue field (as discussed in the previous paragraph) and then applying \cite[III 7.3]{sga1}. Thus we have an elliptic curve $f: \mathcal{E} \to \Spec \mathcal{O}_k$. Let $E_0 = \mathcal{E} \times_{\Spec \mathcal{O}_k} \Spec \bbb{F}_{\ell}$ denote the fiber over the closed point. By smooth-proper base change, \cite[VI \S4 Corollary 4.2]{MilneECbook} $R^if_* \mathbb{Z}/p^n$ is a locally constant sheaf on $\Spec \mathcal{O}_k$ whose stalk at a geometric point $s: \Spec \Omega \to  \Spec \mathcal{O}_k$ is $H^i(\mathcal{E}_s, \mathbb{Z}/p^n)$, where $\mathcal{E}_s = \mathcal{E} \times_{\Spec \mathcal{O}_k} \Spec \Omega$. It follows that the $\Gal_{k}$-action on $H^i(E_{\kbar}, \Z/p^n)$ factors through the quotient map $$\Gal_{k} \to \Gal(k^{\text{un}}/k) \cong \pi_1(\Spec \mathcal{O}_k) \cong \Gal_{\bbb{F}_{\ell^d}},$$ where $k^{\text{un}}$ is the maximal unramified extension of $k$ and $\bbb{F}_{\ell^d}$ is the residue field. It likewise follows that $H^i(E_{\kbar}, \Z/p^n)$ is isomorphic to $H^i((E_0)_{\overline{\bbb{F}_{\ell}}}, \Z/p^n)$ as $\Gal_{\bbb{F}_{\ell^d}}$-modules. This reduces the claim to the case where $k$ is a finite field.

Lastly, suppose that $k$ is a global field. Choose an elliptic curve $E$ over $k$, and let $\nu$ be a place not dividing the discriminant, so $E$ has good reduction at $\nu$. Let $k_{\nu}$ be the completion of $k$ at $\nu$. The decomposition group of $\nu$ in $\Gal_k$ gives an injection $\Gal_{k_{\nu}} \to \Gal_k$, and the $\Gal_{k_{\nu}}$-action on $H^i(E_{\kbar}, \Z/p^n)$ is isomorphic to the $\Gal_{k_{\nu}}$-action on $H^i(E_{\overline{k_{\nu}}}, \Z/p^n)$. By the previous paragraph, we may again reduce  the claim to the case where $k$ is a finite field. 
\end{proof}

We give a slight variation on Theorem \ref{tm:nonexistence}, where only the representation rings $\Rep( \Gal_k, \Z/p)$ with $p$ prime appear.

\begin{tm}\label{tm:nonexistencev2}

Let $k$ be a global field, a local field with infinite Galois group, or a finite field. It is impossible to simultaneously construct all of the following \begin{enumerate}
\item A symmetric monoidal category $\calH_{\Gal_k}(\Spt)$, enriched over abelian groups. Let $\A(\Gal_k)= \End_{\calH_{\Gal_k}(\Spt)}(\Sphere)$ denote the endomorphisms of the symmetric monoidal unit $\Sphere$. 

\item \label{trassumption2} For infinitely many primes $p$, a representation ring $\Rep( \Gal_k, \Z/p)$ with coefficients in $\Z/p$.

\item \label{trassumption3} A symmetric monoidal additive functor $L \Et: \calH_{\bbb{A}^1}(\Spt^{\bbb{P}^1}(k)) \to \calH_{\Gal_k}(\Spt)$, and a homomorphism of rings $\assoc: \A(G) \to \Rep(\Gal_k, \Z/p)$ such that $\assoc \chi (L\Et \Sigma^{\infty}_{\bbb{P}^1}X_+) = \sum_{i=-\infty}^{\infty} (-1)^i H^i_{\et} (X_{\kbar}, \Z/p)$ for smooth proper schemes $X$ over $k$.
\end{enumerate}
\end{tm}

\begin{proof}
Replace Lemma \ref{chirepEFellneq0pn} with Lemma \ref{chirepEFellneq0p} below in the proof of Theorem \ref{tm:nonexistence}.
\end{proof}

\begin{lm}\label{chirepEFellneq0p}
Let $E$ be an elliptic curve over a finite field $\bbb{F}_{\ell}$. Let $p$ be any prime greater than the cardinality of $E(\bbb{F}_{\ell})$. Then $\sum_{i=-\infty}^{\infty} (-1)^i [H^i_{\et} (E_{\kbar}, \Z/p) ]$ is non-zero in $\Rep(\Gal_k,\bbb{Z}/p)$. 
\end{lm}

\begin{proof}

Let $k= \bbb{F}_{\ell}$ and let $F$ be the geometric Frobenius in $\Gal_k$, i.e., the inverse of the arithmetic Frobenius $a \mapsto a^{\ell}$ for all $a$ in $\kbar$. By assumption \eqref{trassumption2}, it suffices to show that $\sum_{i=0}^{2} (-1)^i \Tr F\vert  H^i_{\et} (E_{\kbar}, \Z/p) $ is non-zero in $\Z/p$. By \cite[pg.~292]{MilneECbook} the trace of $F$ equals the trace of the Frobenius endomorphism of $E$. By the Lefschetz Trace Formula,  $$\sum_{i=0}^{2} (-1)^i \Tr F\vert  H^i_{\et} (E_{\kbar}, \Z/p) = \vert E(\bbb{F}_{\ell}) \vert$$ in $\Z/p$  (see \cite[Proposition~7.1]{WIN3-2} for further explanation on the mod $p$ version; it is contained in \cite{sga4andhalf} and can also be deduced from \cite{Hoyois_lef}). Since $p$ is greater than $ \vert E(\bbb{F}_{\ell}) \vert$ by assumption and all elliptic curves contain at least one rational point, it follows that $\sum_{i=0}^{2} (-1)^i \Tr F\vert  H^i_{\et} (E_{\kbar}, \Z/p)$ is non-zero as desired.
\end{proof}

\begin{rmk}
By Section \ref{GR-realization}, when $k=\R$, there does exist a symmetric monoidal category $\calH_{\Gal_k}(\Spt)$ and a symmetric monoidal functor $L \Et: \calH_{\bbb{A}^1}(\Spt^{\bbb{P}^1}(k)) \to \calH_{\Gal_k}(\Spt)$ satisfying the conditions of Theorems \ref{tm:nonexistence} and \ref{tm:nonexistencev2}. 
\end{rmk}

\section{Acknowledgements}
We wish to thank Joseph Ayoub, Gunnar Carlsson, Akhil Mathew,  and all the participants of the AIM workshop on Derived Equivariant Algebraic Geometry, June 13--17, 2016, organized by Andrew Blumberg, Teena Gerhardt, Michael Hill, and Kyle Ormsby, and additionally including the participation of Clark Barwick, Elden Elmanto, Saul Glasman, Jeremiah Heller, Denis Nardin, and Jay Shah. It was a pleasure discussing these ideas with you. It is also a pleasure to thank Serge Bouc for useful correspondence concerning Corollary \ref{co:associdempAg}, and Alexander Duncan for writing and running some GAP code to check an example of this corollary. The second named author also wishes to thank the Institut Mittag--Leffler for their hospitality during their special program ``Algebro-Geometric and Homotopical Methods," where she worked on this paper, and Tom Bachmann for interesting comments.

Jesse Leo Kass was partially sponsored by the Simons Foundation under Grant Number 429929, and the National Security Agency under Grant Number H98230-15-1-0264.  The United States Government is authorized to reproduce and distribute reprints notwithstanding any copyright notation herein. This manuscript is submitted for publication with the understanding that the United States Government is authorized to reproduce and distribute reprints.

Kirsten Wickelgren was partially supported by National Science Foundation Awards DMS-1406380 and DMS-1552730, and an AIM 5-year fellowship. She wishes to thank the Institut Mittag-Leffler for a very pleasant stay at the special program ÒAlgebro-geometric and homotopical methods,Ó while working on this project.
 
} % end of parskip; it started just before the introduction

\bibliographystyle{alpha}

\bibliography{EtaleRealization}

\end{document}